\documentclass{article}

\usepackage{graphicx}%
\usepackage{amsmath,amssymb,amsfonts}%
\usepackage{amsthm}%
\usepackage{xcolor}%

\usepackage{mathtools}
\usepackage{hyperref}
\usepackage{cleveref}
\usepackage{tabularray}

\theoremstyle{definition}%
\newtheorem{theorem}{Theorem}
\newtheorem{observation}[theorem]{Observation}%
\newtheorem{lemma}[theorem]{Lemma}%
\newtheorem{corollary}[theorem]{Corollary}%
\newtheorem{example}[theorem]{Example}%
\newtheorem{definition}[theorem]{Definition}%

\newcommand\NN{\ensuremath{\mathbb{N}}}
\newcommand\QQ{\ensuremath{\mathbb{Q}}}
\newcommand\ZZ{\ensuremath{\mathbb{Z}}}
\newcommand\RR{\ensuremath{\mathbb{R}}}
\newcommand\set[1]{\ensuremath{\{#1\}}}
\newcommand\Infty{\ensuremath{\omega_{0}}}

\newcommand\indexify[1]{\scalebox{0.5}{\ensuremath{#1}}}

\newcommand\V{V}
\newcommand\E{E}
\newcommand\G{G}
\newcommand\func{f}
\newcommand\f{f}
\newcommand\X{X}
\renewcommand\v{v}
\newcommand\va{w}
\newcommand\inv{^{-1}}
\newcommand\m{\mu}
\DeclareMathOperator\RDV{\textup{val}}
\DeclareMathOperator\RDN{\gamma_{\indexify{\textup R}}\negthinspace}
\DeclareMathOperator\nRDN{\bar\gamma_{\indexify{\textup R}}\negthinspace}
\DeclareMathOperator\RDM{\nRDN}
\newcommand\definedas{\coloneqq}
\newcommand\ep{\varepsilon}
\DeclareMathOperator\Neigh{\textup{N}}
\newcommand\intervalcc[1]{\ensuremath{\left[#1\right]}}

\newcommand\intervalco[1]{\ensuremath{\left[#1\right)}}
\newcommand\intervaloo[1]{\ensuremath{\left(#1\right)}}

\renewcommand\implies{\ensuremath{\Rightarrow}}
\newcommand\abs[1]{\ensuremath{\left\lvert#1\right\rvert}}
\renewcommand\emptyset{\varnothing}
\newcommand\powerset[1]{\ensuremath{\raisebox{.4pt}{\ensuremath{\mathcal{P}}}#1}}
\newcommand\D{\textup{d}}
\newcommand{\cupdot}{\mathbin{\mathaccent\cdot\cup}}
\newcommand\x{\relax}
\newcommand\y{\relax}
\newcommand\z{\relax}

\begin{document}

 \title{Roman Domination on Graphings}

 \author{Adrian Rettich (rettich.rptu@use.startmail.com)}
 \date{%
   University of Kaiserslautern-Landau, %
   Department of Mathematics\\%
   Paul-Ehrlich-Stra\ss{}e 14, 67663 Kaiserslautern\\%
   Rheinland-Pfalz, Germany\\[2ex]%
   preprint \today
 }

 \maketitle

 \abstract{\noindent We study a variant of domination, called Roman domination, where we can assign to each vertex a label~\(0\),~\(1\), or~\(2\) and require that every vertex with label~\(0\) has a neighbour with label~\(1\).
   We study the problem of finding a low-cost Roman dominating function on Lebesgue-measurable graphings, that is, on infinite graphs whose vertices are the points of a probability space. We provide a framework to tackle optimisation problems in the measurable combinatorial setting. In particular, we fully answer the Roman domination problem on irrational cycle graphs, a specific type of graphing on the space~\(\RR/\ZZ\) where an irrational number~\(\alpha\) is given and two vertices are adjacent if and only if their distance is~\(\alpha\).
 }

 \setlength{\parindent}{0pt}

\section{Introduction}\label{sec1}

The matter of Borel solutions (or Baire-measurable, or more generally measurable solutions) has been well studied and revealed deep connections in recent years (see for example \cite{borelsurvey}).
The idea, in short, is that given an infinite graph which is embedded into some manner of topological space and a labelling problem%
\footnote{For example~\(k\)-colouring for some fixed integer~\(k\) or, indeed, Roman domination.}%
, we would like to know not only whether the graph admits a solution to the problem at all (for example, whether the graph can be~\(3\)-coloured), but whether it admits a ``nice'' solution (for example, whether it can be~\(3\)-coloured in such a way that each colour class is a Borel set).

One thing which many approaches to the topic have in common is that they consider only decision problems and have no sense of ``scale'', so to speak.
For example, domination problems are of no interest in the Borel context because the only reasonable question to ask is be ``can this graph be dominated by a Borel set'', and the answer is of course yes, since the set of all vertices is a dominating set.
It makes no sense to ask ``how many vertices do I need in order to dominate the graph'' since, at least for locally finite graphs, the answer will always be ``infinitely many''.

We introduce therefore a new approach in which we ask not \emph{whether} a problem an be solved in a measurable way, but optimise the \emph{volume} of vertices needed to solve the problem.
For example, given a graph embedded in a measure space, what is the lowest measure a set can have such that the vertices in that set dominate the entire graph?

We introduce first this new setting, then collect a few general observations and examples that might spark further research.
Finally, we solve the Roman domination problem on the irrational cycle graph.

\section{Preliminaries}

Throughout this paper, we are interested exclusively in the Lebesgue-measurable setting, not the standard Borel setting, and hence  assume the Axiom of Choice.

\begin{definition}
  The set of \emph{natural numbers} is~\(\NN \definedas \set{0, 1, 2, \dots}\).
  Its cardinality is~\(\Infty\).
  A set is called \emph{countable} if its cardinality is no larger than~\(\Infty\).
\end{definition}

\begin{definition}
  Given a set~\(X\), the power set of~\(X\) is denoted by~\(\powerset{X}\).
\end{definition}

\begin{definition}
  A \emph{graph} is a pair~\((\V,\E)\) of sets, where~\(\V\) is called the set of \emph{vertices} and~\(\E\subset\set{x\in\powerset\V\colon \abs{x} = 2}\) is called the set of \emph{edges}.

  In graph theorists' terms, graphs shall be undirected and simple and need not be finite.

  A graph is called \emph{finite} respectively \emph{countable} if its set of vertices is.
\end{definition}

Note that due to our graphs being simple, a finite amount of vertices implies a finite amount of edges, and a countable amount of vertices implies a countable amount of edges.

We also use some standard graph-theoretic notation:~\(\deg\v\) denotes the degree of a vertex, while~\(\Neigh\v\) denotes the set of its neighbours.
For a set of vertices~\(A\), we write~\(G[A]\) for the subgraph induced by~\(A\), that is, the largest subgraph of~\(G\) that contains no vertex from~\(G\setminus A\).
We write further~\(G - A \definedas G[G\setminus A]\).

\begin{definition}
  A graph is called \emph{locally finite} respectively \emph{locally countable} if every vertex has only finitely many respectively countably many neighbours.
\end{definition}

\begin{definition}
  Let~\(\G = (\V,\E)\) be a graph.
  A \emph{graph labelling} for~\(\G\) is a function~\(\func\colon \V\to\X\) into some set~\(\X\).
  If we want to specify the set of possible labels, we can also call this an~\(\X\)-labelling.
\end{definition}

\subsection{Roman Domination}

We introduce first the finite version of the main problem which we are to tackle in this paper, \emph{Roman domination}.

An introduction to Roman domination (and a historical background) can be found in \cite{cockayne}.

\begin{definition}
  Let~\(\G = (\V,\E)\) be a graph.
  A \emph{Roman dominating function} for~\(\G\) is a~\(\set{0,1,2}\)-labelling~\(\func\)
  such that
  \[\forall\v\in\V : \func\v = 0 \implies \exists\va\in\Neigh\v : \func\va = 2.\]
  We call~\(\RDV f \definedas \sum_{\v\in\V}\func\v\) the \emph{value} of the roman domination scheme~\(\func\).
  The \emph{Roman domination number}~\(\G\), written~\(\RDN\G\) is the smallest integer~\(i\) such that a roman domination scheme of value~\(i\) exists for~\(\G\), or~\(\infty\) if no scheme with finite value exists.
\end{definition}

Now, on a graph with infinitely many vertices, that number will mostly be infinite. We shall therefore introduce a slightly modified version of this problem in the next section.

\begin{theorem}
  The Roman domination number of a graph~\(\G\) is finite if and only if~\(\G\) admits a finite dominating set.
\end{theorem}
\begin{proof}
  Let~\(\G\) be a graph. If~\(\G\) is finite, then the statement is tautological.
  Assume hence that~\(\G\) has at least~\(\Infty\) many vertices.

  If~\(\G\) admits a finite dominating set~\(\X\), assign the number~\(2\) to each vertex in~\(\X\) and~\(0\) to every other vertex.
  Since~\(\X\) is finite, this yields a roman domination scheme with finite value.

  If, on the other hand,~\(\G\) admits a roman domination scheme~\(\f\) with finite value, then the set~\(\f\inv\set{1,2}\) must be a finite dominating set.
\end{proof}

\begin{observation}
  Infinite but locally finite graphs cannot admit a finite dominating set\footnote{Proving this is a simple exercise.} and hence never have finite Roman domination number.
\end{observation}

The final piece of terminology we need from the final case is the following.

\begin{definition}
  A finite graph is called \emph{Roman} if it admits a Roman dominating function~\(f\) of minimum value with~\(f\inv1 = \emptyset\).
\end{definition}

\subsection{Graphings}

We pass now from arbitrary infinite graphs to those embedded in a measure space, though we introduce only the necessary basics.
An in-depth look at graphings can be found in \cite{hombook}.

\begin{definition}
  Let~\((\V,\E)\) be a graph,let~\(\v\in\V\), and let~\(A\subset\V\).
  We denote by
  \[
    \deg_{A}\v
  \]
  the number of neighbours of~\(\v\) in~\(A\).
\end{definition}

\begin{definition}
  Let~\(X = (\V,\mathcal A)\) be a Borel space, and let~\(\m\) be a probability measure on~\(X\).
  A \emph{graphing} on~\((X,\m)\) is a graph~\(\G = (\V,\E)\) such that~\(\E\) is Borel when viewed as a subset of the product space~\(X\times X\) and such that for any two sets~\(A,B\in\mathcal A\) we have
  \begin{equation}
    \int_{A}\deg_{B}\D\m = \int_{B}\deg_{A}\D\m.\label{edgemeasure}
  \end{equation}

  Note the double occurrence of the set~\(\V\), once as the ground set for our probability space and once as the set of vertices of our graph.
\end{definition}

We shall abuse notation extensively by referring to the tuple~\((\V,\E)\) as a graphing when the sigma algebra and the measure are clear from context and by writing~\(\v\in\G\) when really we mean~\(\v\in\V\).

Graphings are usually considered as Borel graphs, but it turns out that the edge measure condition \ref{edgemeasure} already ensures that if we consider the completion\footnote{The completion of a measure space is the smallest measure space containing it as a subspace which is complete, that is, in which every null set is measurable.}~\(\mathcal{A}'\) of~\(\mathcal A\), we have~\(\forall A\in\mathcal{A}' : \Neigh(A) \in \mathcal{A}'\) (\cite{hombook}, exercise 18.32).

\begin{definition}
  To emphasise that we are talking about the Lebesgue measurable setting (in particular, our measure is complete and our sigma algebra need not be the Borel sigma algebra), we shall refer to such graphings as \emph{probability graphs}.
\end{definition}

\section{Measurable Roman Domination}

\begin{definition}
  Let~\(\G = (\V,\E)\) be a probability graph.
  A Roman dominating function~\(\f\) on~\(\G\) is called \emph{measurable} if the sets~\(\f\inv2\) and~\(\f\inv1\) are both measurable.
\end{definition}

Note: this implies that~\(\f\inv0\) is measurable as well.

\begin{definition}
  Let~\(\G = (\V,\E)\) be a probability graph, and let~\(\f\) be a measurable Roman dominating function for~\(\G\).
  The \emph{measured value} of~\(\f\) is
  \[
    \RDV\f \definedas 2\cdot\m\f\inv2 + \m\f\inv1.
  \]
  We call
  \[
    \RDM\G \definedas \inf_{\f : \f \textup{ is a measurable Roman dominating function for \G}} \RDV\f
  \]
  the \emph{Roman domination measure} of~\(\G\).
\end{definition}

The following is immediately clear.

\begin{observation}
  The measured value of any measurable Roman dominating function is between~\(0\) and~\(2\).
  The Roman domination measure of any probability graph is between~\(0\) and~\(1\).
\end{observation}

Put differently, Roman domination measure tells us what percentage of vertices we have to use (weighted by their label) in order to dominate the graph.

In order to compare our new definition to the finite case, we slightly modify the latter.

\begin{definition}
  Let~\(\G = (\V,\E)\) be a graph.
  The \emph{normalised Roman domination number} of~\(\G\) is
  \[
    \nRDN\G \definedas \frac{\RDN\G}{\abs\V}.
  \]
\end{definition}

We have purposely used the same notation as for the Roman domination measure, since the definitions are compatible in a natural way.

\begin{lemma}
  Let~\(\G = (\V,\E)\) be a finite graph endowed with the normalised counting measure, that is, turned into a probability graph by prescribing~\(\mathcal A \definedas \powerset\V\) and
  \[
    \forall X\in\powerset\V : \m{}X \definedas \frac{\abs X}{\abs\V}.
  \]
  Then its normalised Roman domination number (as a finite graph) is equal to its Roman domination measure (as a probability graph).
\end{lemma}
\begin{proof}
  Again, this is immediate from the definitions.
\end{proof}

One distinction of which we need to be aware in the measurable setting is that the Roman domination measure, being an infimum,  may not actually be attainable.

\begin{definition}
  A probability graph is called \emph{Mithraic}\footnote{After the Roman cult that baffles historians to this day.} if it does not admit a Roman dominating function of minimum measured value.
\end{definition}

In particular, finite graphs are never Mithraic.

We expand our definition of Roman graphs to include probability graphs.

\begin{definition}
  A probability graph is called \emph{Roman} if it is non-Mithraic and admits a Roman dominating function~\(f\) of minimum measured value with~\(f\inv1 = \emptyset\),
  or if it is Mithraic and for every~\(\ep{}>0\), a Roman dominating function~\(f_{\ep{}}\)  with~\(f_{\ep{}}\inv1 = \emptyset\) can be constructed whose measured value exceeds the graph's Roman domination measure by no more than~\(\ep{}\).
\end{definition}

Note that a finite graph is Roman as a finite graph if and only it is Roman as a probability graph when endowed with the normalised counting measure.

We immediately see that we can ignore null sets for this definition.

\begin{lemma}
  Let~\(\G\) be a probability graph that admits a Roman dominating function~\(\f\) of minimum measured value such that~\(\m{}\f\inv1 = 0\) or that is Mithraic and  for every~\(\ep{}>0\), admits a Roman dominating function~\(f_{\ep{}}\) with~\(\m{}f_{\ep{}}\inv1 = 0\) whose measured value exceeds the Roman domination measure of~\(G\) by no more than~\(\ep{}\).
  Then~\(\G\) is Roman.
\end{lemma}
\begin{proof}
  As in the finite case, replacing ones by twos does not change whether any vertex is dominated.
  Hence, we need only check measurability.

  But~\(\f\inv2\) and~\(\f\inv1\) are both measurable by definition of a Roman dominating function.
  And since~\(\f\inv1\) is null, the measured value does not change.
\end{proof}

\section{Roman Domination is not Domination}

We give an example to show that the existence of the label~\(1\) is not merely smoke and mirrors, but actually changes the situation even in the infinite setting.

Of course, graphs with enough isolated vertices are an immediate example, but we need not go that far.

\begin{example}
  Consider the probability space~\(\intervalco{0,1}\) with the Lebesgue measure.
  Prescribe then that a vertex pair~\(v, w\) shall be connected by an edge if and only if~\(v+\frac{1}{4} \in \set{w,w+1}\).
  In other words, our graph is made out of uncountably many cycles of length~\(4\).
  This is a graphing.

  Suppose we wanted to find a Roman dominating function for this graph which does not use the label~\(1\).
  Then in each connected component, we would need to label at least two vertices with~\(2\), thus labelling at least half the vertices and producing a value of at least~\(2\cdot\frac{1}{2} = 1\).

  This is not optimal: labelling each vertex in~\(\intervalco{0,\frac{1}{4}}\) with~\(2\), each vertex in~\(\intervalco{\frac{1}{2},\frac{3}{4}}\) with~\(1\), and the rest with~\(0\) produces a Roman dominating function of measured value~\(2\cdot\frac{1}{4} + \frac{1}{4} = \frac{3}{4}\).
\end{example}

\section{The Main Result}

In this chapter, we compute the Roman domination measure of the irrational cycle graph.
We have chosen this example because it is a well-known counterexample in the theory of Borel colouring.

\subsection{Irrational Cycle Graphs}

The underlying probability space for this chapter is the one-dimensional manifold~\(S^{1}\) normalised to have a circumference of~\(1\) and endowed with the Lebesgue measure.
The graph theorist uncomfortable with manifolds may instead imagine the interval~\(\intervalco{0,1}\) endowed with the Lebesgue measure and glued together such that going over the edge wraps around to the other side.

Another useful way to think about this space is as~\(\RR/\ZZ\), endowed with the quotient measure, which is what we shall use for most computations in this section.

To decide where the edges are, choose an~\(\alpha{}\in\intervaloo{0,1}\setminus\QQ\).
Two vertices are connected by an edge if and only if the length of the arc between them is equal to~\(\alpha{}\), or in other words,~\(v\) and~\(w\) are connected if and only if~\(\abs{v - w} = \alpha\) in~\(\RR/\ZZ\).

The resulting graph, which we shall call~\(G_{\alpha{}}\), is~\(2\)-regular, and because~\(\alpha{}\) is irrational, it contains no finite cycles.
In particular, it is~\(2\)-colourable\footnote{This can be seen, for example, via the De Bruijn-Erd\H{o}s theorem (\cite{bruijn}).}.

But now comes the interesting discovery: even though two colours, say red and blue, suffice to properly colour the vertices of~\(\G_{\alpha{}}\), there is no such colouring where the set~\(\X \definedas \set{\v\in\G_{\alpha{}} : \v \textup{ is blue}}\) is measurable (or even Borel).
A deeper delve into this surprising result can be found in \cite{borelsurvey}.
It turns out that in order to have the colour classes be measurable, one needs at least three colours -- one more than the chromatic number of~\(\G_{\alpha{}}\).
Three is then called the \emph{measurable chromatic number} of~\(\G_{\alpha{}}\), and finding the measurable chromatic number of infinite graphs is its own fascinating field of research, see for example \cite{conley2016}.

\subsection{Roman Domination of~\(\G_{\alpha{}}\)}

In order to work with the irrational cycle graph, we shall need the following technical lemma.

\begin{lemma}\label{alphaGoesEverywhere}
  Let~\(x\in\intervalcc{0,1}\), and let~\(\alpha{}\in\intervaloo{0,1}\setminus\QQ\).

  Then for every~\(\ep{} > 0\), there exists a~\(k\in\NN\) such that~\(k\alpha{}\in \intervaloo{x - \ep{}, x+\ep{}}\) (in~\(\RR/\ZZ\)).
\end{lemma}
\begin{proof}
  This follows for example from the fact that~\(\set{n\alpha{}}_{n\in\NN}\) is a low-discrepancy sequence, as shown in \cite{niederreiter}.
\end{proof}

For the remainder of this text, we shall quietly assume that all our computations are done in~\(\RR/\ZZ\) instead of mentioning this fact every time.

\renewcommand\x{G_{\alpha}}
\renewcommand\y{\intervalco{0,\ep}}
\renewcommand\z{e}
\begin{theorem}\label{schemeTheorem}
  Let~\(\ep{} > 0\).
  Then there exists a Roman dominating function~\(f\) for~\(G_{\alpha{}}\) with~\(\RDV f < \frac{2}{3} + \ep{}\) and~\(f\inv1 = \emptyset\).

  In particular, the Roman domination measure of any irrational cycle graph is no larger than~\(\frac{2}{3}\).
\end{theorem}
\begin{proof}
  Fix an arbitrary~\(\ep > 0\).
  We construct a dominating set~\(X\) for~\(G_{\alpha}\) with measure~\(\frac{1}{3} + c\cdot\ep\), where~\(c\) is a constant.
  Assigning the label~\(2\) to each vertex in~\(X\) and~\(0\) to all other vertices then implies the claim.

  We first define the ``distance'' function
  \[
    d : G_{\alpha}\to\NN, v \mapsto \min\set{k\in\NN : v-k\alpha \in \intervalco{0,\ep}},
  \]
  where the interval is, as always, understood to live in~\(\RR/\ZZ\).
  This is well-defined due to \cref{alphaGoesEverywhere}. 

  We now claim the following:
  let~\(v\in \x\) with~\(dv \neq 0\).
  Then~\(d(v - \alpha) = dv - 1\).

  Indeed: if~\(dv \neq 0\), then~\(v\notin\y\).
  Thus we either have~\(v-\alpha\in\y\) and thus~\(d(v-\alpha) = 0\) and~\(dv = 1\), or we have~\(d(v-\alpha) = k\) for some~\(k > 0\), in which case~\(dv = d(v-\alpha) + 1\) because~\(v\notin\y\).

  We now set~\(X \definedas d\inv(3\NN)\).

  To see that~\(X\) is measurable, notice that for any~\(n\in\NN\), we have
  \[d\inv\set{0,\dots,n} = \y + \set{0,\alpha, 2\alpha,\dots,n\alpha},\] which is measurable.
  Thus for any~\(n\in\NN_{\indexify{>0}}\) the set
  \[d\inv\set{n} = d\inv\set{0,\dots,n}\setminus d\inv\set{0,\dots,n-1}\]
  is measurable, whence 
  \[X = \bigcup_{n\in\NN}d\inv\set{3n}\]
  is measurable.


  To show that~\(X\) dominates~\(\x\), pick any~\(v\in\x\).
  We distinguish three cases:
  If~\(v\in \y\), then~\(v\in X\) and thus~\(X\) dominates~\(v\).
  If~\(v\in \y - \alpha\), then a neighbour of~\(v\) is in~\(\y\) and thus~\(X\) dominates~\(v\).

  It remains to examine the case where~\(v\) is neither in~\(\y\) nor in~\(\y - \alpha\).
  This means that~\(dv \neq 0\) and~\(d(v+\alpha) \neq 0\), which by our earlier claim implies that~\(d(v-\alpha) = dv - 1\) and~\(dv = d(v+\alpha) - 1\).
  Hence,~\(d\set{v-\alpha,v,v+\alpha} \cap 3\NN \neq \emptyset\), whereby~\(X\) must dominate~\(v\).

  It remains to compute the Lebesgue measure of~\(X\).

  Notice first that any~\(v\in X\cap X+\alpha\) must fulfil~\(dv \in 3\NN\) and~\(d(v+\alpha) \in 3\NN\), which is only possible if~\(d(v+\alpha) = 0\) or, in other words,~\(v \in \y-\alpha\).
  Similarly, any~\(v\in X \cap X-\alpha\) must fulfil~\(dv \in 3\NN\) and~\(d(v-\alpha) \in 3\NN\), which is only possible if~\(dv = 0\) or, in other words,~\(v \in \y\).

  From this, we conclude that the sets~\(X\),~\(X - \alpha\), and~\(X + \alpha\) are disjoint outside of~\(\y \cup \y - \alpha\), which is used in the following estimate.
  We set~\(\z \definedas \y \cup \y-\alpha\) and write~\(\cupdot\) to signify a disjoint union.
  The translation invariance of the Lebesgue measure then yields
  \begin{align*}
    3\m X - 3\m\z &= \m(X-\alpha) - \m\z + \m X-\m\z + \m(X+\alpha) - \m\z \\
                  &\leq \m((X-\alpha)\setminus\z) + \m(X\setminus\z) + \m((X+\alpha)\setminus\z) \\
                  &= \m((X-\alpha)\setminus\z \cupdot X\setminus\z \cupdot (X+\alpha)\setminus\z)\\
                  &= \m(\x \setminus \z)\\
                  &= \m\x - \m\z.
  \end{align*}
  Rearranging yields~\(3\m X \leq \m\x + 2\m\z \leq 1 + 4\m\y = 1 + 4\ep\) and thus~\(\m X \leq \frac{1}{3} + \frac{4}{3}\ep\), as promised.
\end{proof}

We also establish a lower bound on the Roman domination measure of~\(G_{\alpha{}}\).

\begin{lemma}\label{parsprototo}
  Let~\(\G = (\V,\E)\) be a probability graph of finite maximum degree~\(\Delta\).
  Then for every measurable set~\(A\subset\V\), we have
  \[
    \m\Neigh A \leq (\Delta+1)\m A.
  \]
\end{lemma}
\begin{proof}
  Let~\(\G\) be a probability graph with finite maximum degree~\(\Delta\), and let~\(A\) be a measurable set of vertices of~\(\G\).
  By the edge measure condition of graphings, we have
  \begin{align*}
    \m(\Neigh(A)\setminus A) &= \int_{\Neigh(A)\setminus A}1\D\m\\
                             &\leq \int_{\Neigh(A)\setminus A}\deg_{A}\D\m\\
                             &=\int_{A}\deg_{\Neigh(A)\setminus A}\D\m\\
    &\leq \int_{A}\deg \D\m
  \end{align*}
  and thus
  \begin{align*}
    \m(\Neigh(A)) &\leq \int_{A}\deg \D\m + \m A\\
    &\leq\int_{A}\Delta \D\m + \m A\\
                  &=(\Delta+1)\m A,
  \end{align*}
  as promised.
\end{proof}

\begin{theorem}
  No irrational cycle graph can be covered by a Roman dominating function of measured value less than~\(\frac{2}{3}\).
\end{theorem}
\begin{proof}
  Let~\(\f\) be a measurable Roman dominating function for~\(G_{\alpha{}}\).
  Due to \cref{parsprototo}, we have
  \begin{align*}
    \m{}f\inv0 \leq 2\m{}f\inv2,
  \end{align*}
  and because every vertex must be labelled, we know that
  \begin{align*}
    1 &= \m{}f\inv0 + \m{}\f\inv1 + \m{}f\inv2 \\
      &\leq \m{}\f\inv1 + 3\m{}f\inv2.
  \end{align*}
  Minimising~\(\m{}\f\inv1 + 2\m{}\f\inv2\) with this constraint yields the unique optimal solution
  \begin{align*}
    \m{}\f\inv1 = 0,\quad \m{}\f\inv2 = \frac{1}{3},
  \end{align*}
  as claimed.
\end{proof}

Combining this with \cref{schemeTheorem} yields the following.

\begin{corollary}
  The Roman domination measure of the irrational cycle graph is~\(\frac{2}{3}\).
\end{corollary}

The equation in the proof leads to another nice corollary.

\begin{corollary}
  Every irrational cycle graph is Roman.
\end{corollary}
\begin{proof}
  We have already established that~\(\frac{2}{3}\) is a lower bound on the Roman domination measure of~\(G_{\alpha{}}\).

  If this bound is sharp, by the optimisation in the preceding proof, the vertices labelled 1 are a null set, and hence the value of the Roman dominating function does not change if we label all of them 2 instead.
  Just as in the finite case, this is still a valid Roman dominating function.

  If the bound is not sharp, we have already shown how to explicitly construct an Roman dominating function of value arbitrarily close to~\(\frac{2}{3}\) with only labels 0 and 2.
\end{proof}

\subsection{\(G_{\alpha{}}\) is Mithraic}

We have established that~\(G_{\alpha{}}\) has Roman domination measure~\(\frac{2}{3}\).
In this section, we show that this optimal value cannot be attained.

Since~\(G_{\alpha{}}\) is Roman, we shall henceforth assume without loss of generality that any Roman dominating function of value~\(\frac{2}{3}\) contains no vertices labelled 1 at all.

We use the following shorthand for some of our proofs.

\begin{definition}
  Let \(\alpha\in\RR\).
  We write \(T_{\alpha}: \RR/\ZZ\to\RR/\ZZ, x\mapsto x+\alpha\).
\end{definition}

To show that \(G_{\alpha}\) is Mithraic, we establish several ``forbidden'' configurations for any Roman dominating function of minimal measured value, then show that no measurable function avoiding said configurations exist.

\begin{definition}
  Let~\(f\) be a Roman dominating function for~\(G_{\alpha{}}\).
  A vertex is called \emph{social} (under~\(f\)) if it has label 2 and at least one of its neighbours has label 2.
\end{definition}

We first show that it is impossible to have more than a null set of social vertices in a Roman dominating function of measure~\(\frac{2}{3}\).

\begin{lemma}\label{socialVerticesAreMeasurable}
  The set of social vertices is measurable.
\end{lemma}
\begin{proof}
  As we know, the set~\(f\inv2\) is measurable.
  Because~\(T_{\alpha{}}\) is a translation, the sets~\(T_{\alpha{}}f\inv2\) and~\(T_{-\alpha{}}f\inv2\) are measurable.

  Hence their union~\(X \definedas T_{\alpha{}}f\inv2 \cup T_{-\alpha{}}f\inv2\) is measurable.

  Which vertices does this set contain?
  It contains all zeroes, since each of them must be to the left or to the right of a 2.

  It also contains all social vertices for the same reason.

  It contains no unsocial 2s since an unsocial 2 can by definition not be adjacent to another two.

  Thus, the set of social vertices is equal to~\(X\setminus f\inv\set{0}\), which is a measurable set.
\end{proof}

\begin{lemma}\label{socialVerticesAreNull}
  Let~\(f\) be a Roman dominating function for~\(G_{\alpha{}}\) of value~\(\frac{2}{3}\).
  Then the set of all social vertices is a null set.
\end{lemma}
\begin{proof}
  Let~\(f\) be a~\(\frac{2}{3}\) value Roman dominating function for~\(G_{\alpha{}}\) and suppose the set~\(S\) of all social vertices (which is measurable by \cref{socialVerticesAreMeasurable}) had measure~\(\ep{}>0\).

  We now bound the volume of vertices of label 0.

  The set~\(U\) of all unsocial vertices can cover at most two zeros per vertex, or a total (due to \cref{parsprototo}) of~\(2\cdot\m{}U\).

  Each social vertex can cover at most one zero, or a total of~\(1\cdot\m{}S\).

  That means that the total amount of zeroes covered is no larger than
  \begin{align*}
    2\cdot\m{}U + 1\cdot\m{}S &= 2\cdot(\frac{1}{3}-\ep{}) + \ep{} \\
                        &= \frac{2}{3} - \ep{} \\
                        &< \frac{2}{3},
  \end{align*}
  a contradiction.
\end{proof}

There is another kind of vertices that we would like to avoid.

\begin{definition}
  Let~\(f\) be a Roman dominating function for~\(G_{\alpha{}}\). A vertex is called \emph{needy} (under~\(f\)) if it has label 0 and both of its neighbours have label 2.
\end{definition}

\begin{lemma}\label{needyVerticesAreMeasurable}
  The set of needy vertices is measurable.
\end{lemma}
\begin{proof}
  The set~\(f\inv2\) is measurable.

  Thus the set~\(T_{\alpha{}}f\inv2\) is measurable.

  Hence the set~\(X \definedas f\inv2 \cup T_{\alpha{}}f\inv2\) is measurable.
  This set, let us call it~\(A\), contains all vertices except those whose label is 0 and whose left neighbour's label is also 0.
  Hence that set must be measurable.

  Rotating that set by~\(-\alpha{}\) yields the set~\(B\) of all 0s whose right neighbour's label is also 0.
  Hence this set is measurable as well.

  Thus, the union~\(A\cup B\) is measurable and contains exactly all vertices whose label is 0 and who are not needy.

  Thus the set of all needy vertices is~\(f\inv0 \setminus (A\cup B)\), a measurable set.
\end{proof}

\begin{lemma}\label{needyVerticesAreNull}
  Let~\(f\) be a Roman dominating function for~\(G_{\alpha{}}\) of value~\(\frac{2}{3}\).
  Then the set of all needy vertices is a null set.
\end{lemma}
\begin{proof}
  This is essentially analogous to the proof of \cref{socialVerticesAreNull}.

  Let~\(f\) be a~\(\frac{2}{3}\) value Roman dominating function for~\(G_{\alpha{}}\) and suppose the set~\(N\) of all needy vertices (which is measurable by \cref{needyVerticesAreMeasurable}) had measure~\(\ep{}>0\).

  We now bound the volume of vertices of label 0.

  Let~\(M\) denote the set of vertices adjacent to a needy vertex, necessarily a subset of~\(f\inv2\).
  Due to~\(T_{\alpha{}}\) being a translation,~\(M\) is measurable.

  Set~\(Q \definedas f\inv2\setminus M\), also measurable.

  The set~\(Q\)  can cover at most two zeros per vertex, or a total (due to \cref{parsprototo}) of~\(2\cdot\m{}Q\).

  Each  vertex in~\(M\) can cover two zeroes as well, but any pair of two adjacent vertices in~\(M\) can cover at most three 0s in total.
  Since vertices in~\(M\) always come in pairs, this gives a total of~\(\frac{3}{2}\cdot\m{}M\).

  Because each vertex in~\(N\) must be adjacent to at least one vertex in~\(M\),  the total amount of zeroes covered is no larger than
  \begin{align*}
    2\cdot\m{}Q + \frac{3}{2}\cdot\m{}M &\leq 2\cdot(\frac{1}{3}-\ep{}) + \frac{3}{2}\cdot \ep{} \\
                                  &= \frac{2}{3} - 2\ep{} +  \frac{3}{2}\ep{}\\
                                  &= \frac{2}{3} -  \frac{\ep{}}{2}\\
                                  &< \frac{2}{3},
  \end{align*}
  a contradiction.  
\end{proof}

\begin{lemma}\label{socialComponentsAreNull}
  Let~\(f\) be a Roman dominating function for~\(G_{\alpha{}}\) of value~\(\frac{2}{3}\).

  Then  only a null set of connected components can contain any social vertices.
\end{lemma}
\begin{proof}
  Let~\(f\) be a Roman dominating function for~\(G_{\alpha{}}\) of value~\(\frac{2}{3}\).

  Let us fix any collection~\(C\) of connected components such that each of them contains a social vertex.

  By \cref{socialVerticesAreNull}, the set of those vertices must have measure zero.

  Each connected component is countable, hence each element of~\(c\in C\) can be written as~\(\bigcup_{n\in\NN}v^{c}_{n}\), where~\(v^{c}_{n}\) are the vertices of~\(c\).

  Thus we have
  \begin{align*}
    \bigcup_{c\in C}c = \bigcup_{c\in C}\bigcup_{n\in\NN}v^{c}_{n}
    = \bigcup_{n\in\NN}\bigcup_{c\in C}v^{c}_{n}
  \end{align*}
  and thus
  \begin{align*}
    \m{}\bigcup_{c\in C}c &= \m{}\bigcup_{n\in\NN}\bigcup_{c\in C}v^{c}_{n} \\
                       &= \sum_{n\in\NN}\m{}\bigcup_{c\in C}v^{c}_{n} \\
                       &=0,
  \end{align*}
  as promised.
\end{proof}

\begin{lemma}\label{needyComponentsAreNull}
  Let~\(f\) be a Roman dominating function for~\(G_{\alpha{}}\) of value~\(\frac{2}{3}\).

  Then only a null set of connected components can contain any needy vertices.
\end{lemma}
\begin{proof}
  This is entirely analogous to the proof of \cref{socialComponentsAreNull}.
\end{proof}

We combine our lemmas to show the promised result.

\begin{theorem}
  Every irrational cycle graph is Mithraic.
\end{theorem}
\begin{proof}
  Suppose we had a Roman dominating function~\(f\) for~\(G_{\alpha{}}\) of value~\(\frac{2}{3}\) which was not socially awkward.

  Then by  \cref{socialComponentsAreNull} and \cref{needyComponentsAreNull}, we know that all but a null set of the connected components are free of social and needy vertices.

  Since~\(f\) is a Roman dominating function for each connected component, this means that in every connected component outside some null set, labels occur in the order 2, 0, 0, 2, 0, 0, and so forth \emph{ad infinitum}.

  In particular, rotating the set of vertices of label 2 in those components by~\(3\alpha{}\) yields again the same set, or put differently, the set of those vertices is fixed under the map that rotates a point by~\(3\alpha{}\).

  As shown in \cite{ergodic}, proposition 4.2.1, that map is ergodic. (Actually, Viana and Oliveira show that the rotation by~\(2\alpha{}\) is ergodic, but since they show it for any irrational~\(\alpha{}\) and~\(\frac{3}{2}\) is rational, it can immediately be adapted for our case.)

  Since the rotation is ergodic and fixes all but a null set of the vertices labelled 2, the set of all vertices labelled 2 must have measure~\(0\) or~\(1\), neither of which is equal to~\(\frac{2}{3}\), leading to a contradiction.
\end{proof}

\section{Conclusion}

We have only begun to explore optimisation problems on probability graphs.
The preceding section solves the Roman domination problem (and indeed, regular domination as well) on irrational cycle graphs.
A generalisation to the class of all \(2\)\nobreakdash-regular graphings is possible and will appear in a follow-up paper, but it is currently unclear whether this approach can be generalised to higher degrees.

Furthermore, other domination problems could be extended to the measurable case in the same way.
As in the finite case, the question arises on which classes of probability graphs the Roman domination problem is equivalent to regular domination (that is, which probability graphs are Roman), and of course whether this classification only depends on finite subgraphs or on some manner of emergent property.

\vspace{.5cm}

\textbf{Acknowledgements.}
The author would like to thank Prof.~Dr.~Sven~O.\ Krumke for his guidance and patience. They further thank Markus Kurtz for their input on the current proof of \cref{schemeTheorem}, which replaces an earlier, much clunkier version.


\end{document}